\documentclass{amsart}

\usepackage{amsmath}
\usepackage{amssymb}
\usepackage{xcolor}
\usepackage{graphicx}
\usepackage{placeins}
\usepackage{hyperref}
\hypersetup{hidelinks}

\newtheorem{theorem}{Theorem}
\newtheorem{corollary}[theorem]{Corollary}
\theoremstyle{definition}
\newtheorem{remark}[theorem]{Remark}

\newcommand{\Z}{\mathbb Z}
\newcommand{\Q}{\mathbb Q}
\newcommand{\M}{\operatorname{Mod}}
\newcommand{\bdry}{\partial}
\newcommand{\DT}[1]{t_{#1}}
\allowdisplaybreaks[4]

\begin{document}

\title[Handlebody monodromy]
{Lefschetz fibrations with handlebody monodromy}

\author[R. \.{I}. Baykur]{R. \.{I}nan\c{c} Baykur}
\address{Department of Mathematics and Statistics, University of Massachusetts, Amherst, MA 01003, USA}
\email{inanc.baykur@umass.edu}

\author[S. Hirose]{Susumu Hirose}
\address{Department of Mathematics, Tokyo University of Science, Noda, Chiba 278--8510, Japan}
\email{hirose\_susumu@ma.noda.tus.ac.jp}


\begin{abstract}
We study relatively minimal Lefschetz fibrations over closed oriented
surfaces whose monodromy extends over a fixed 3--dimensional handlebody.
If $X$ is the total space and $\ell$ is the number of critical points,
we prove that $\sigma(X)=-\ell$.  Such fibrations with $\ell>0$ exist
precisely for fiber genus $g\geq3$ and base genus $h\geq2$, and one can
take $\ell=1$ for every such pair.  We also show that hyperelliptic handlebody monodromy forces the
fibration to be a surface bundle and that every meridian twist, separating
or nonseparating, has commutator length two in the handlebody group for
$g\geq3$.
\end{abstract}

\maketitle

\section{Introduction}

Let $V_g$ be a genus-$g$ handlebody, $g\geq0$, obtained from a 3--ball
by attaching $g$ 1--handles, with oriented boundary $\bdry V_g=\Sigma_g$.
We regard
its mapping class group $\M(V_g)$ as the subgroup of $\M(\Sigma_g)$ consisting
of mapping classes which extend over $V_g$.  An essential simple closed curve on
$\Sigma_g$ which bounds a properly embedded disk in $V_g$ is called a
\emph{meridian}, and a Dehn twist about such a curve is called a
\emph{meridian twist}.  We refer to \cite{Hensel2020} for background on
handlebody groups.  We use right-handed Dehn twists, write products in
functional order, and set $[x,y]=xyx^{-1}y^{-1}$.  Thus
$\varphi\DT{c}\varphi^{-1}=\DT{\varphi(c)}$.

A \emph{Lefschetz fibration} is a smooth map from an oriented 4--manifold
to an oriented surface, a submersion except at finitely many critical points,
near which it is given by $(z_1,z_2)\mapsto z_1z_2$ in orientation-preserving
complex coordinates. Our fibrations have closed, connected fiber and base
surfaces and are relatively minimal, that is, no fiber contains an exceptional
sphere.  We allow the critical set to be empty; such a fibration
is a \emph{surface bundle} (over a surface).  A Lefschetz fibration $f\colon X\to\Sigma_h$ has
\emph{handlebody monodromy} if, after identifying a regular fiber with
$\Sigma_g$, its monodromy lies in $\M(V_g)$.  Equivalently, it admits a
monodromy factorization
\begin{equation}\label{eq:monodromy-factorization}
 \DT{c_1}\cdots \DT{c_\ell}\,
 [A_1,B_1]\cdots[A_h,B_h]=1
 \qquad\text{in }\M(V_g).
\end{equation}
Every vanishing cycle $c_i$ is essential by relative minimality, and hence is
a meridian by McCullough's theorem \cite{McCullough2006}.

Our main result determines the possible fiber and base genera for Lefschetz
fibrations with critical points and handlebody monodromy, and shows that their
signatures are given by a simple closed formula.

\begin{theorem}\label{thm:main}
Let $f\colon X\to\Sigma_h$ be a genus-$g$ Lefschetz fibration with
handlebody monodromy and $\ell$ critical points.  Then
\begin{equation}\label{eq:signature-main}
 \sigma(X)=-\ell.
\end{equation}
Moreover, for fixed $g$ and $h$, such a fibration exists for some
$\ell>0$ if and only if $g\geq3$ and $h\geq2$.
For every such pair, one can take $\ell=1$.
\end{theorem}

When $\ell=0$, Theorem~\ref{thm:main} recovers the vanishing of the
signature for surface bundles whose monodromy extends over a handlebody.
This case also follows from \cite[Theorem~A]{GiansiracusaTillmann2011}
for $g\geq2$ and from \cite[Lemma~4.2]{KunoSato2020} for $g\geq1$.
Such surface bundles are nevertheless abundant. For example, when
$g,h\geq1$, one may send one standard generator of the base fundamental
group to a meridian twist and all the others to the identity.

The signature formula also gives an obstruction to holomorphicity.

\begin{corollary}\label{cor:nonholomorphic}
A Lefschetz fibration with handlebody monodromy and $\ell>0$ cannot be
 holomorphic or a fiber sum of holomorphic Lefschetz
fibrations.
\end{corollary}


Next, for $g\geq1$, fix a hyperelliptic involution $\iota$ on $\Sigma_g$
which extends to an involution of $V_g$.  Let $H(\Sigma_g)$ be its
centralizer in $\M(\Sigma_g)$, and set
\[
 \mathcal H(V_g):=H(\Sigma_g)\cap\M(V_g).
\]
This is the \emph{hyperelliptic handlebody group}
\cite{HiroseKin2017,KunoSato2020}.  A Lefschetz fibration has
\emph{hyperelliptic handlebody monodromy} if its monodromy lies in
$\mathcal H(V_g)$ after identifying a regular fiber with $\Sigma_g$.

\begin{corollary}\label{cor:hyperelliptic}
Every Lefschetz fibration with hyperelliptic handlebody monodromy is a surface
bundle.
\end{corollary}

We also obtain the following group-theoretic consequence.

\begin{corollary}\label{cor:commutator-length}
For $g\geq3$, every meridian twist, separating or nonseparating, has
commutator length two in $\M(V_g)$.
\end{corollary}

Our proofs combine the Meyer functions of Kuno--Sato and Endo
\cite{KunoSato2020,Endo2000}, Li's constraints on Lefschetz fibrations
\cite{Li2000}, and a handlebody version of the lantern construction of
Korkmaz--Ozbagci in \cite{KorkmazOzbagci2001}.

\medskip
\noindent\textit{Acknowledgements.}
The first author was partially supported by NSF grant DMS-2506431 and a Simons
Foundation Travel Grant. The second author was supported by JSPS KAKENHI Grant Number JP24K06746. This work was initiated during the first author's
visit to Japan, supported by JSPS KAKENHI Grant JP22K18267.
LLM tools were used for editing and proofreading; the authors take full responsibility for the mathematical arguments and results.

\section{The signature formula}

First suppose $\ell=0$. The relation in~\eqref{eq:monodromy-factorization}
allows us to construct an oriented $V_g$--bundle $W\to\Sigma_h$.
Its boundary is a surface bundle $Y\to\Sigma_h$ with the same monodromy
as $f$, and $\sigma(Y)=0$ since $Y=\bdry W$.
For $g\geq2$, surface bundles are determined by their monodromy
\cite{EarleEells1969}, so $X$ and $Y$ are isomorphic as bundles.
For $g=0,1$, their restrictions to the complement of a disk in the base
are still isomorphic, since this complement retracts to a graph.
Thus $X$ is obtained from $Y$ by removing a fiber neighborhood
$\Sigma_g\times D^2$ and regluing it by a fiber-preserving diffeomorphism
of its boundary. For $g=1$, this is a logarithmic transform of
multiplicity one; for $g=0$, the possible nontrivial regluing is a
Gluck twist \cite{Gluck1962}.  Since the pieces being glued are unchanged, Novikov
additivity gives $\sigma(X)=\sigma(Y)=0$ in either case.

For the rest of the proof, assume $\ell>0$. Relative minimality then
implies $g\geq1$, since every simple closed curve on $S^2$ is
nullhomotopic. Let
$\rho\colon\M(\Sigma_g)\to\operatorname{Sp}(2g,\Z)$ be the symplectic
representation and let $\tau_g$ be Meyer's signature cocycle. Kuno and
Sato define $\phi_g^V(\varphi)=\sigma(M_\varphi)$ for $\varphi\in\M(V_g)$,
where the mapping torus
\[
 M_\varphi=([0,1]\times V_g)/\bigl((0,p)\sim(1,\varphi(p))\bigr)
\]
has the product orientation, as in \cite[Section~3.1]{KunoSato2020}.
Here $\varphi$ also denotes a representative diffeomorphism. For a compact
oriented 4--manifold $M$, possibly with boundary, the signature is the number
of positive eigenvalues minus the number of negative eigenvalues of its
intersection form on $H_2(M;\Q)$; zero eigenvalues are ignored. In the above
mapping torus, following the interval from $0$ to $1$ takes $[0,p]$ to
$[1,p]=[0,\varphi^{-1}(p)]$, so its monodromy is $\varphi^{-1}$.
Kuno and Sato prove
\begin{equation}\label{eq:coboundary}
 \tau_g\bigl(\rho(\varphi),\rho(\psi)\bigr)
 =\phi_g^V(\varphi)+\phi_g^V(\psi)-\phi_g^V(\varphi\psi)
\end{equation}
for $\varphi,\psi\in\M(V_g)$ \cite[Lemma~4.2]{KunoSato2020}.

For a meridian $c$, the homological formula of Kuno and Sato gives
\begin{equation}\label{eq:V-values}
 \phi_g^V(\DT{c})=
 \begin{cases}
  1,& c\text{ is nonseparating},\\
  0,& c\text{ is separating}.
 \end{cases}
\end{equation}
For a standard nonseparating meridian, the homological action has one
nontrivial block $\left(\begin{smallmatrix}1&1\\0&1\end{smallmatrix}\right)$.
Their formula then gives a form with matrix $\operatorname{diag}(1,0,\ldots,0)$,
whose signature is one \cite[Lemma~4.4]{KunoSato2020}. All nonseparating
meridians are conjugate in $\M(V_g)$, so this proves the first case.
For a separating meridian the homological action is the identity, and the
same formula gives the zero form \cite[Theorem~1.1]{KunoSato2020}.

After a small perturbation, assume that the critical values are distinct.
Remove disjoint disk neighborhoods of them, leaving a surface bundle
$X^\circ\to B$. Since $B$ retracts to a graph, the handlebody monodromy
extends this bundle to a $V_g$--bundle $W\to B$, oriented by the base
followed by the fiber. The induced orientation on each boundary circle
of $B$ is clockwise around the corresponding critical value, so its
monodromy is $\DT{c_i}^{-1}$. With the mapping-torus convention above,
the part of $\bdry W$ lying over $\bdry B$ is the disjoint union of the
$M_{\DT{c_i}}$, and the remaining boundary is $X^\circ$.
After rounding corners, $\bdry W$ has signature zero, so Novikov
additivity gives
\begin{equation}\label{eq:smooth-complement-signature}
 \sigma(X^\circ)=-\sum_{i=1}^{\ell}\phi_g^V(\DT{c_i}).
\end{equation}
This is the boundary argument of \cite[Lemma~4.2]{KunoSato2020}, applied
to the punctured base.

Let $n$ and $s$ be the numbers of nonseparating and separating vanishing
cycles. Equation~\eqref{eq:V-values} gives $\sigma(X^\circ)=-n$.
A singular-fiber neighborhood has signature zero for a nonseparating
vanishing cycle and $-1$ for a separating one \cite{Ozbagci2002}.
Another application of Novikov additivity gives
\[
 \sigma(X)=-n-s=-\ell,
\]
which proves \eqref{eq:signature-main}.

\begin{remark}\label{rem:lagrangian}
The formula $\sigma(X)=-\ell$ also holds if the monodromy preserves a
Lagrangian subspace $L\subset H_1(\Sigma_g;\Q)$.  Let $v=[c]\neq0$ be the
class of a nonseparating vanishing cycle.  Since the transvection $T_v$
preserves $L$, for every $u\in L$ we have
\[
 T_v(u)=u\mathbin{\pm}\langle u,v\rangle v\in L.
\]
If $v\notin L$, this forces $\langle u,v\rangle=0$ for every $u\in L$, so
$v\in L^\perp=L$, a contradiction.  Thus every nonzero vanishing class lies
in $L$; separating vanishing cycles have zero homology class.  Choose an
integral symplectic basis whose first $g$ vectors span $L$ over $\Q$.  The
homological monodromy then lies in the upper block triangular subgroup, on
which Kuno--Sato's formula cobounds $\tau_g$
\cite[Remark~3.8]{KunoSato2020}.  The same homological calculation as above
gives value $1$ on a nonseparating vanishing twist and $0$ on a separating
one.  Applying the coboundary identity to a pants decomposition gives
$\sigma(X^\circ)=-n$, and the singular-fiber neighborhoods contribute $-s$.
For $\ell=0$, the coboundary evaluates to zero on the closed base.
\end{remark}

\section{Hyperelliptic handlebody monodromy}

Assume $g\geq1$, as in the definition of hyperelliptic handlebody monodromy.

Endo's hyperelliptic Meyer function
\[
 \phi_g^H\colon H(\Sigma_g)\longrightarrow\Q
\]
satisfies \eqref{eq:coboundary} with $\phi_g^H$ in place of $\phi_g^V$ on
$H(\Sigma_g)$.  Subtracting the two coboundary identities shows that
\begin{equation}\label{eq:lambda}
 \lambda_g:=\phi_g^H-\phi_g^V\colon\mathcal H(V_g)\longrightarrow\Q
\end{equation}
is a homomorphism on $\mathcal H(V_g)$ \cite{KunoSato2020}.

Let $c$ be a meridian with $\DT{c}\in\mathcal H(V_g)$.  The identity
$\iota\DT{c}\iota^{-1}=\DT{\iota(c)}=\DT{c}$ implies that $c$ is invariant
under $\iota$ up to isotopy.  Endo's function has value
$(g+1)/(2g+1)$ on a nonseparating twist and $-4k(g-k)/(2g+1)$ on a
separating twist of type $k$, where the complementary genera are $k$ and
$g-k$ \cite{Endo2000}; see also \cite[Section~4.1]{Kuno2016}.
Subtracting \eqref{eq:V-values} gives
\begin{equation}\label{eq:lambda-values}
 \lambda_g(\DT{c})=
 \begin{cases}
  -\dfrac{g}{2g+1},&c\text{ is nonseparating},\\[7pt]
  -\dfrac{4k(g-k)}{2g+1},
    &c\text{ separates }\Sigma_g\text{ into genera }k\text{ and }g-k.
 \end{cases}
\end{equation}
In particular, every positive meridian twist in $\mathcal H(V_g)$ has strictly
negative image under $\lambda_g$.

\begin{proof}[Proof of Corollary~\ref{cor:hyperelliptic}]
Let $f$ have hyperelliptic handlebody monodromy.  Since $\lambda_g$ is a
homomorphism to an abelian group, it vanishes on commutators.  If $\ell>0$,
applying it to the monodromy factorization of $f$ gives
\begin{align*}
 0=\lambda_g(1)
 &=\lambda_g\bigl(\DT{c_1}\cdots\DT{c_\ell}
 [A_1,B_1]\cdots[A_h,B_h]\bigr)\\
 &=\sum_{i=1}^\ell\lambda_g(\DT{c_i})<0,
\end{align*}
a contradiction.  Hence $\ell=0$, so $f$ is a surface bundle.
\end{proof}

For $g=1,2$, the hyperelliptic involution is central in $\M(\Sigma_g)$,
so $\mathcal H(V_g)=\M(V_g)$.  Thus every Lefschetz fibration with
handlebody monodromy in these genera is a surface bundle.

\section{Proof of the main theorem}

\begin{proof}[Proof of Theorem~\ref{thm:main}]
It remains to prove the existence statement.  Suppose $\ell>0$.  By the
preceding section, $g\geq3$.  If the base is $S^2$, then
\eqref{eq:signature-main} gives $\sigma(X)=-\ell$.  Li rules out
this equality in \cite[Proposition~4.1]{Li2000}: it would
force $X$ to be a blow-up of an $S^2$--bundle over $\Sigma_g$, contradicting
his genus bound for Lefschetz fibrations on ruled surfaces.

If the base is $T^2$, equip $X$ with a symplectic form compatible with the
fibration.  We have
\[
 e(X)=\ell,\qquad \sigma(X)=-\ell,\qquad
 c_1^2(X)=2e(X)+3\sigma(X)=-\ell<0.
\]
Relative minimality and the positive genus of the base imply that $X$ is
minimal \cite[Corollary~3.3]{Li2000}.  Li's symplectic Parshin--Arakelov
inequality then gives $c_1^2(X)\geq0$ \cite[Theorem~3.4]{Li2000}, a
contradiction.

We now construct examples for every $g\geq3$ and $h\geq2$.  Consider the
lantern embedding in $\bdry V_g$ shown in Figure~\ref{fig:lantern}(a).  The
neck curves $a,b,c,d$ bound disjoint disks $D_a,D_b,D_c,D_d$ separating
$V_g$ into a lower 3--ball and an upper genus-$(g-3)$ handlebody.  The
lantern is the four-holed sphere on the boundary of the lower ball, so
$x,y,z$ also bound disks in this ball.  All seven curves are therefore
meridians.  With right-handed Dehn twists and products written in functional
order, the lantern relation is
\begin{equation}\label{eq:lantern-handlebody}
 \DT{a}\DT{b}\DT{c}\DT{d}=\DT{x}\DT{y}\DT{z}
 \qquad\text{in }\M(V_g).
\end{equation}
This is the handlebody version of the lantern construction in
\cite{KorkmazOzbagci2001}.

Let $D_x,D_y,D_z$ be disks bounded by $x,y,z$ in the lower ball.  Cutting
along either pair $(D_b,D_x)$ or $(D_y,D_c)$ leaves a connected handlebody
of genus $g-2$: both pieces of the cut-open lower ball remain attached to
the upper handlebody.  The ordered pairs are therefore equivalent under
$\M(V_g)$, as one sees by completing them to disk systems and matching the
paired disk faces on the resulting balls.  Also, $D_d$ and $D_z$ are
nonseparating.  Hence there are
$\varphi,\psi\in\M(V_g)$ such that
\[
 \varphi(D_b)=D_y,\qquad \varphi(D_x)=D_c,
 \qquad\text{and}\qquad \psi(D_z)=D_d.
\]
Since $[\DT{x}\DT{b}^{-1},\varphi]=\DT{x}\DT{b}^{-1}\DT{y}\DT{c}^{-1}$
and $[\DT{z},\psi]=\DT{z}\DT{d}^{-1}$, the lantern relation gives
$\DT{a}=[\DT{x}\DT{b}^{-1},\varphi][\DT{z},\psi]$, or equivalently
\begin{equation}\label{eq:one-node-factorization}
 \DT{a}[\psi,\DT{z}][\varphi,\DT{x}\DT{b}^{-1}]=1
 \qquad\text{in }\M(V_g).
\end{equation}
This is a monodromy factorization for a genus-$g$ Lefschetz fibration over
$\Sigma_2$ with $\ell=1$.  Adding $h-2$ trivial commutators gives a fibration
over every $\Sigma_h$, $h\geq2$.  The vanishing cycle $a$ is nonseparating,
so the fibration is relatively minimal.  This completes the proof.
\end{proof}

\begin{figure}[tbp]
\centering
\begin{minipage}[t]{.485\textwidth}
\centering
\includegraphics[width=\linewidth]{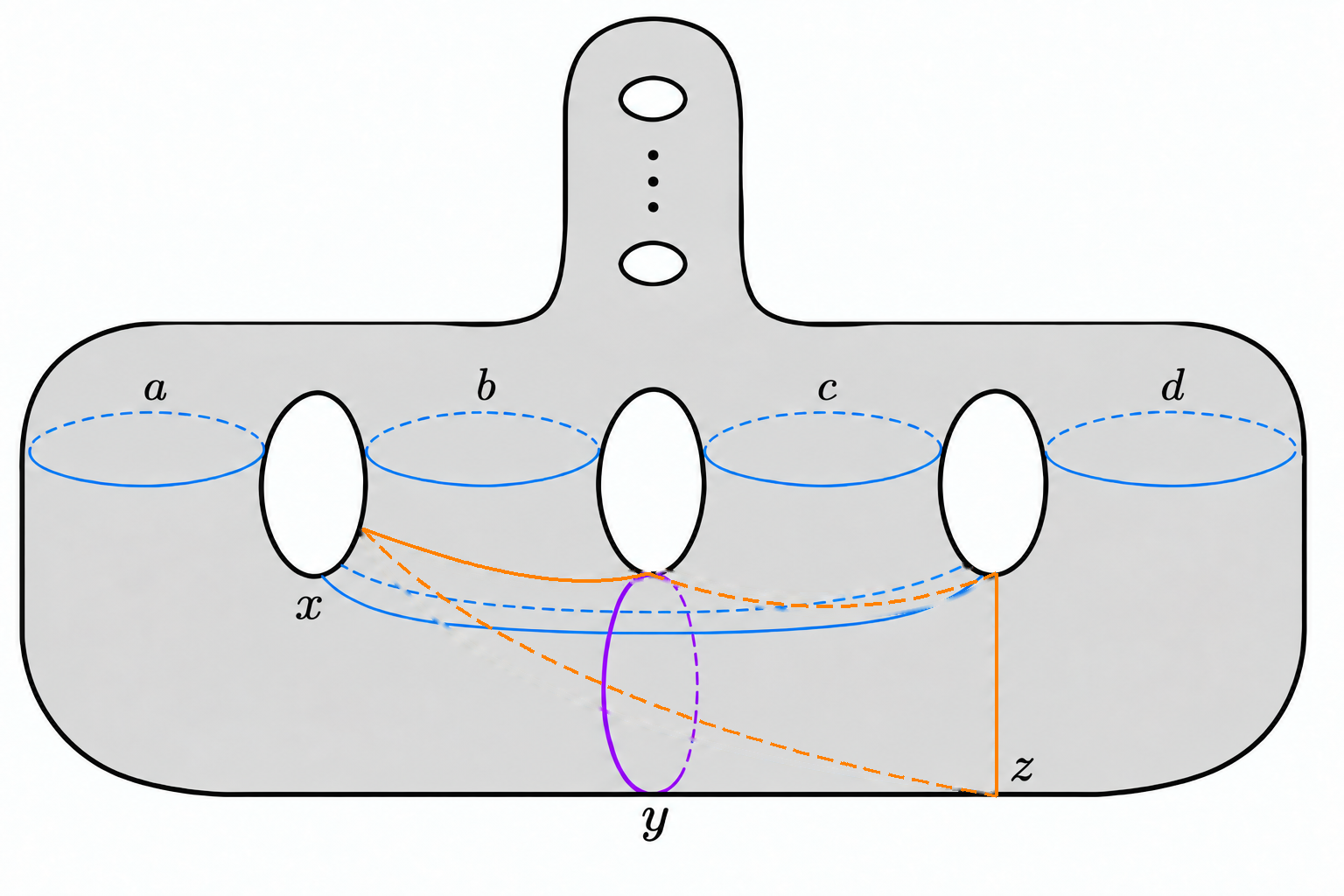}
\par\smallskip
{\small (a)}
\end{minipage}\hfill
\begin{minipage}[t]{.485\textwidth}
\centering
\includegraphics[width=\linewidth]{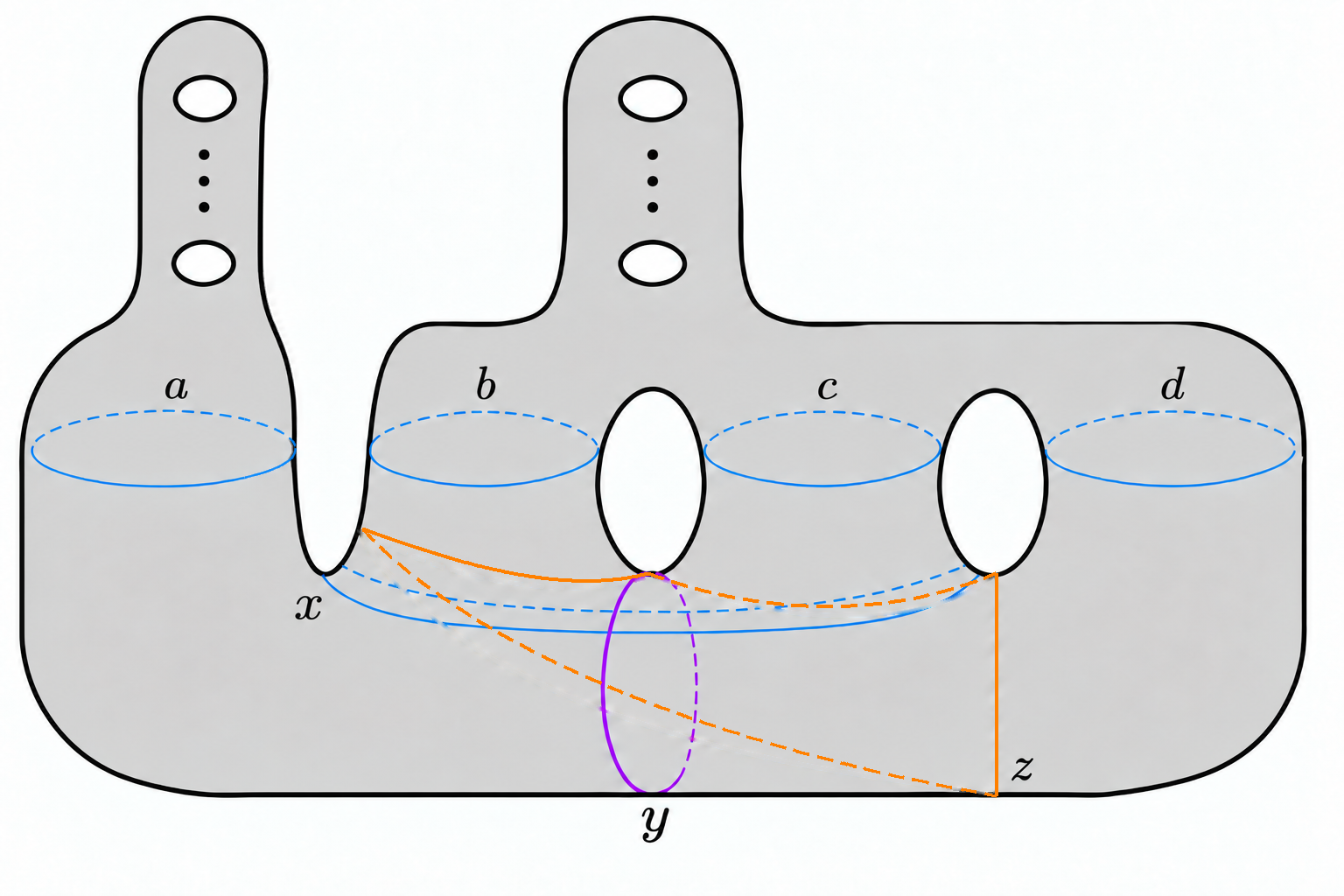}
\par\smallskip
{\small (b)}
\end{minipage}
\caption{Lantern embeddings in $\bdry V_g$.
(a) The curve $a$ is nonseparating, and the upper part represents $g-3$
additional handles.
(b) The curve $a$ separates genera $k$ and $g-k$; the upper parts on the
left and right represent $k$ and $g-k-2$ handles, respectively.
An upper part with zero handles is omitted. All seven curves are meridians
in both cases.}
\label{fig:lantern}
\end{figure}

\FloatBarrier
\begin{proof}[Proof of Corollary~\ref{cor:nonholomorphic}]
By Theorem~\ref{thm:main}, $g\geq3$ and $h\geq2$. Since
$e(X)=4(g-1)(h-1)+\ell$, the signature formula gives
\[
 c_1^2(X)=2e(X)+3\sigma(X)=8(g-1)(h-1)-\ell.
\]
If $f$ were holomorphic, the Parshin--Arakelov inequality would give
$c_1^2(X)\geq8(g-1)(h-1)$ \cite[Section~1]{Li2000}, a contradiction.
Every summand in a fiber sum decomposition of $f$ inherits handlebody
monodromy. Thus any holomorphic summand would have to be a surface bundle,
and so would their fiber sum.
\end{proof}

\begin{remark}\label{rem:critical-points}
For fixed $g\geq3$ and $h\geq2$, the calculation in the preceding proof,
together with Li's inequality $c_1^2(X)\geq2(g-1)(h-1)$
\cite[Theorem~3.4]{Li2000}, gives
\begin{equation}\label{eq:critical-point-bound}
 \ell\leq6(g-1)(h-1).
\end{equation}

At least, any integer $1\leq\ell\leq h-1$ can be realized. Pull back the
one-critical-point genus-$g$ fibration over $\Sigma_2$ along a connected
unbranched cover $\Sigma_{\ell+1}\to\Sigma_2$ of degree $\ell$. The pullback
has exactly $\ell$ critical points and remains relatively minimal with
handlebody monodromy. Adding $h-\ell-1$ trivial commutators to its monodromy
factorization gives the desired base genus.
\end{remark}

\begin{proof}[Proof of Corollary~\ref{cor:commutator-length}]
For a nonseparating meridian, relation~\eqref{eq:one-node-factorization}
gives the upper bound two, since all nonseparating meridians are conjugate
under $\M(V_g)$.

Now let the meridian be separating, with complementary genera $k$ and
$g-k$, where $1\leq k\leq\lfloor g/2\rfloor$.  In
Figure~\ref{fig:lantern}(b), put $k$ handles above $a$ and
$g-k-2\geq0$ above the other part.  The curve $a$ then has the required
separating type.  The lower 3--ball is attached through $D_a$ to a genus-$k$
handlebody and through $D_b,D_c,D_d$ to a genus-$(g-k-2)$ handlebody.  The interior lantern curves again bound disks in the lower ball, so all seven curves are meridians.

In the lantern, $x$ separates $\{b,c\}$ from $\{a,d\}$, $y$ separates
$\{a,b\}$ from $\{c,d\}$, and $z$ separates $\{a,c\}$ from $\{b,d\}$.  After cutting along $D_b,D_x$, the two pieces
of the lower ball are joined through the $c$- and $d$-necks to the same
upper handlebody; after cutting along $D_y,D_c$, they are joined through
the $b$- and $d$-necks.  The genus-$k$ piece remains attached through $a$.
Thus both complements are connected handlebodies of genus $g-2$.  Also,
cutting along $D_d$ or $D_z$ leaves a connected complement through the
$b$- and $c$-necks. As in the proof of Theorem~\ref{thm:main}, the
disk-system argument gives $\varphi,\psi$ with the prescribed images, and hence
\[
 \DT{a}[\psi,\DT{z}][\varphi,\DT{x}\DT{b}^{-1}]=1
 \qquad\text{in }\M(V_g).
\]
Every separating meridian of type $k$ is conjugate to $a$ in $\M(V_g)$,
by identifying the complementary handlebodies.  Its twist therefore has
commutator length at most two.

Finally, a meridian twist cannot be a single commutator: a relation
$\DT{c}[A,B]=1$ would give a Lefschetz fibration over $T^2$ with handlebody
monodromy and one critical point, contrary to Theorem~\ref{thm:main}.
Since a meridian twist is not the identity, its commutator length is two.
\end{proof}

\begin{remark}\label{rem:positive-length}
Let $\phi\in\M(V_g)$, $g\geq3$, admit a nonempty positive factorization
into meridian twists. Its commutator length $r$ is finite by
Corollary~\ref{cor:commutator-length}. Closing any such factorization using
$r$ commutators for $\phi^{-1}$ gives a Lefschetz fibration over $\Sigma_r$,
where $r\geq2$ by Theorem~\ref{thm:main}. The factorization length therefore
satisfies \eqref{eq:critical-point-bound} with $h=r$.

For a fixed open book, or more generally, a spinal open book with pages of genus $g\geq3$, the Stein
fillings arising from allowable Lefschetz fibrations whose capped
fibrations are relatively minimal and have handlebody monodromy have
bounded Euler characteristics; see \cite{BaykurVanHornMorris2015} for the
compatible Stein structures. For a capped fibration $\widehat X\to B$,
the boundary signature formula and $|\phi_g^V|\leq g$
\cite[Theorem~1.1]{KunoSato2020} give a uniform bound on
$\sigma(\widehat X)+\ell$. Attaching a fixed surface-bundle cap and
applying Li's inequality then bounds $\ell$, and hence the Euler
characteristic.
\end{remark}

\FloatBarrier
\bibliographystyle{amsplain}
\bibliography{refs}

\end{document}